\documentclass[11pt]{article}
\usepackage{enumerate}
\usepackage{amssymb,a4wide,latexsym,makeidx,epsfig}
\usepackage{amsthm}
\usepackage{dsfont}
\usepackage{amsmath}
\usepackage{lipsum, mathtools}
\usepackage{enumerate,enumitem}
\setlist{nosep}
\usepackage{mathrsfs}
\usepackage{xcolor}
\usepackage{tikz}
\usepackage{setspace}
\usepackage{geometry}
\usepackage{appendix}
\usepackage{multirow}
\usepackage[colorlinks=true, linkcolor=black, anchorcolor=black, citecolor=black, urlcolor=black, CJKbookmarks=true]{hyperref}

\usepackage{microtype}
\usepackage{colonequals}
\usepackage{graphics,float,import}
\usepackage{bbm}

\allowdisplaybreaks

\newtheorem{theorem}{Theorem}[section]

\newtheorem{lemma}[theorem]{Lemma}

\newtheorem{proposition}[theorem]{Proposition}

\newtheorem{problem}[theorem]{Problem}

\newtheorem{construction}[theorem]{Construction}

\newtheorem{claim}[theorem]{Claim}

\newcommand{\ex}{{\rm ex}}

\newcommand{\eps}{\varepsilon}
\newcommand{\tG}{\widetilde{G}}
\def\V{\mathcal{V}}

\newenvironment{proofclaim}[1][Proof of claim]{\begin{proof}[#1]}{\end{proof}}

\numberwithin{equation}{section}

\begin{document}
	\textwidth 150mm \textheight 225mm
	
	\title{Complete tripartite subgraphs of balanced tripartite graphs with large minimum degree}
	\author{
		Yihan Chen\footnote{School of Mathematical Sciences, University of Science and Technology of China, Hefei 230026, China.}~~~~~~
		Jialin He\footnote{School of Mathematical Sciences,  Key Laboratory of MEA (Ministry of Education), Shanghai Key Laboratory of PMMP,  and Nantong Institute for Applied Mathematics and Artificial Intelligence, East China Normal University, Shanghai 200241, China.}~~~~~~
		Allan Lo\footnote{School of Mathematics, University of Birmingham, B15 2TT, United Kingdom.}~~~~~~
		Cong Luo\footnote{School of Mathematics and Computational Science, Xiangtan University, Xiangtan, Hunan 411105, China}~~~~~~
		Jie Ma\footnotemark[1]~~~~~~~
		Yi Zhao\footnote{Department of Mathematics and Statistics, Georgia State University, Atlanta, GA 30303.}
	}
	\date{\today}
	\maketitle
	\begin{center}
		\begin{minipage}{120mm}
			\begin{center}
				{\small {\bf Abstract}}
			\end{center}
			{\small In 1975, Bollob\'{a}s, Erd\H{o}s, and Szemer\'{e}di \cite{bollobas1975complete} asked what minimum degree guarantees an octahedral subgraph $K_3(2)$ in any tripartite graph $G$ with $n$ vertices in each vertex class. We show that $\delta(G)\geq n+2n^{\frac{5}{6}}$ suffices, thus improving the bound $n+(1+o(1))n^{\frac{11}{12}}$ of Bhalkikar and Zhao \cite{bhalkikar2023subgraphs} obtained by following the approach of \cite{bollobas1975complete}. Bollob\'{a}s, Erd\H{o}s, and Szemer\'{e}di conjectured that $n+cn^{\frac{1}{2}}$ suffices and there are many $K_3(2)$-free tripartite graphs $G$ with $\delta(G)\geq n+cn^{\frac{1}{2}}$. {In the appendix} we prove this conjecture under the additional assumption that every vertex in $G$ is adjacent to at least $(1/5+\varepsilon)n$ vertices in any other vertex class. 
			}
		\end{minipage}
	\end{center}
	
	\section{Introduction}
	As a foundation stone of extremal graph theory, the celebrated Tur\'{a}n's theorem in 1941~\cite{turan1941external} determined the maximum size $\ex(n,K_t)$ of graphs $G$ of order $n$ without $K_t$ as a subgraph (the $t=3$ case is also known as Mantel's theorem~\cite{mantel1907opgaven28}). In 1974, Bollob\'{a}s, Erd\H{o}s, and Straus~\cite{bollobas1974complete} obtained a Tur\'{a}n-type result that determined the maximum size of an $r$-partite graph  not containing $K_t$ as a subgraph for any $r\ge t$. In 1975, Bollob\'{a}s, Erd\H{o}s, and Szemer\'{e}di~\cite{bollobas1975complete} investigated the $r=t$ case of the following problem.
	\begin{problem}\label{problem1}
		Given integers $n$ and $3 \leq t \leq r$, what is the largest minimum degree $\delta(G)$ among all $r$-partite graphs $G$ with parts of size $n$ and which do not contain a copy of $K_t$?
	\end{problem}
	The $r=t$ case has received a lot of attention and has found applications in linear arboricity, hypergraph matching, list coloring, among others (see~\cite{Graf2020Finding} and references therein). Graver (see~\cite{bollobas1975complete}) answered Problem~\ref{problem1} for $r=t=3$ and Jin~\cite{jin1992complete} solved it for $r=t=4, 5$. 
	The $r=t$ case of Problem~\ref{problem1} was finally settled by {Haxell~\cite{haxell2001note}}, Haxell and Szab\'{o}~\cite{haxell2006odd} and Szab\'{o} and Tardos~\cite{szabo2006extremal}. Recently Lo, Treglown, and Zhao~\cite{lo2022complete} solved many $r > t$ cases of the problem, including when $r \equiv -1$ $(\text{mod } t - 1)$ and $r\ge (3t-4)(t-2)$. For more related results, we refer interested readers to~\cite{alon1988linear, haxell2001note, haxell2006odd, jin1992complete, szabo2006extremal}.
	
	\vskip 0.1in
	
	Let $G_r(n)$ be an (arbitrary) \emph{balanced $r$-partite graph} with parts of size $n$, and let $K_r(s)$ denote the complete $r$-partite graph with parts of size $s$. {The aforementioned result of Graver states that every graph $G=G_3(n)$ with $\delta(G)\ge n+1$ contains a copy of $K_3$ (and this is best possible).} Bollob\'{a}s, Erd\H{o}s, and Szemer\'{e}di~\cite{bollobas1975complete} asked when $G_3(n)$ contains $K_3(2)$ (known as the \emph{octahedral graph}) as a subgraph.
	\begin{problem}\label{prob2}
		What minimum degree $\delta(G)$ in a graph $G=G_3(n)$ ensures a subgraph $K_3(2)$?
	\end{problem}
	The authors of~\cite{bollobas1975complete} conjectured that $\delta(G)\geq n+cn^{\frac{1}{2}}$ is sufficient for some constant $c$.
	They also stated that for $n\geq 2^8$, $\delta(G)\geq n+2^{-\frac{1}{2}}n^{\frac{3}{4}}$ guarantees a $K_3(2)$, which follows from their result on~$K_3(s)$~\cite[Theorem 2.6]{bollobas1975complete}. Unfortunately, Bhalkikar and Zhao~\cite{bhalkikar2023subgraphs} found a miscalculation in the proof of~\cite[Theorem 2.6]{bollobas1975complete}. After correcting this error, they followed the approach in~\cite{bollobas1975complete} and obtained the following result.
	\begin{theorem}[\cite{bhalkikar2023subgraphs}]\label{thmzhao}
		Given $s\geq 2$, $\eps>0$, and sufficiently large $n$, every graph $G = G_3(n)$ with $\delta(G) \geq n + (1+\eps)(s-1)^{1/(3s^2)} n^{1-1/(3s^2)}$ contains a copy of $K_3(s)$. In particular, $\delta(G)\geq n+(1+\eps)n^{\frac{11}{12}}$ guarantees a~$K_3(2)$.
	\end{theorem}
	
	In this paper, we establish several results towards Problem~\ref{prob2}. 
	First we improve Theorem~\ref{thmzhao} as follows.

		\begin{theorem}\label{thmk3s}
		Given $s\ge 2$, $C= 2(s-1)^{\frac{1}{s+1}}$, and sufficiently large $n$, every graph $G = G_3(n)$ with 
		$\delta(G)\geq n+Cn^{1-\frac{1}{s(s+1)}}$ contains a copy of $K_3(s)$.
		In particular, if $\delta(G)\geq n+2n^{\frac{5}{6}}$, then $G$ contains a~$K_3(2)$.
	\end{theorem}
	
		This is a strengthening of \cite{bhalkikar2023subgraphs}, as the way to fix~\cite{bollobas1975complete}. Note that we can easily replace the constant $2$ in $C$ by $1+o(1)$ but use the current definition of~$C$ for simplicity.
	
	Let $V_1,V_2,V_3$ be the three vertex classes of a graph $G=G_3(n)$. 
	We call $\min\{d_G(v,V_i): v\in V(G)\setminus V_i, i\in [3]\}$ the {\it minimum partial degree} of $G$,
	where $d_G(v,V_i)$ denotes the number of neighbors of $v$ in $V_i$. Partial degree {(also known as \emph{partite degree} and \emph{local degree})} conditions have been widely studied in multipartite settings, for instance in the study of multipartite Hajnal--Szemer\'{e}di theorem as well as Fischer's conjecture \cite{han2013On, lo2013Amulti}, and the existence of independent transversal (factors) conditioned on local degree~\cite{glock2022average, haxell2024degree, kang2022colorings, loh2007Independent,yuster2021factors}.
	Motivated by these works, we establish a result which shows that $G=G_3(n)$ contains a copy of $K_3(2)$ under the conjectured condition $\delta(G)\geq n+cn^{\frac{1}{2}}$ if, in addition, the minimum partial degree is at least $(1/5+\varepsilon)n$. The precise statement and proof are given in Appendix~\ref{Sec:Appendix A}.
	
	Bhalkikar and Zhao~\cite{bhalkikar2023subgraphs} also constructed many non-isomorphic $K_3(2)$-free tripartite graphs with minimum degree at least $n+n^{\frac{1}{2}}$. In Section~\ref{Section: Constructions}, we construct new $K_3(2)$-free tripartite graphs with minimum degree $n+(1-o(1))n^{\frac{1}{2}}$. 

	For a graph $G$, let $T(G)$ denote the number of triangles in $G$. For $n,t\in \mathbb{N}$, let $f(n,t)$ denote the minimum $T(G)$ over all $G=G_3(n)$ of minimum degree $\delta(G)$ at least $n+t$. Bollob\'{a}s, Erd\H{o}s, and Szemer\'{e}di~\cite{bollobas1975complete} also studied $f(n,t)$. They showed that $f(n,1)=4$ for $n\geq 4$, and for general $t$, they proved $f(n,t)\geq t^3$. For $t\leq n/5$, they constructed tripartite graphs $G=G_3(n)$ with $\delta(G)\geq n+t$ and $T(G)= 4t^3$, which implies $f(n,t)\leq 4t^3$. 
	Based on this, they asked that for $t\leq n/5$, is it true that $f(n,t)\geq 4t^3$? We give a simple proof of $f(n,t) \ge  n^2 (3t-n)/2$, which improves $f(n,t)\geq t^3$ when $\frac{\sqrt{3}-1}{2}n\leq t\leq n$, and gives the exact value of $f(n,t)$ for even $n$ and $t \ge n/2$. 	
	\begin{proposition}\label{2t^3 triangles}
		For $1 \le t \le n$, $f(n,t) \ge  n^2 (3t-n)/2$ and equality holds if $n$ is even and $t \ge n/2$. 
	\end{proposition}
	
	\subsection{Notation}
	Recall that $G=G_3(n)$ is an (arbitrary) balanced tripartite graph with parts $V_1,V_2,V_3$ such that $|V_i|=n$ for $i=1,2,3$. Let $V(G)$ denote the vertex set of~$G$ and $E(G)$ denote the edge set of~$G$. 
	We consider the subscript of $V_i$ to be \emph{modulo~$3$ with values $1, 2, 3$, instead of $0, 1, 2$}.

	Given $v \in V (G)$, we write $N(v)$ for the neighborhood of $v$ and define $d_G(v) = |N(v)|$ as the degree of~$v$ in~$G$. Let $\delta(G)$ and $e(G)$ denote the minimum degree and the number of edges of~$G$ respectively. 
	We often view~$G$ as an oriented graph with edges directed from~$V_i$ to~$V_{i+1}$.
	For $v \in V_i$, let $N_G^+(v)$ (resp.~$N_G^-(v)$) be the set of vertices in~$V_{i +1}$ (resp. $V_{i- 1}$) that are joined to $v$.  Let $d_G^+(v)=|N_G^+(v)|$ and $\delta^+(G)=\min\limits_{v\in V(G)}|N_G^+(v)|$.
	We define $d_G^-(v)$ and $\delta^-(G)$ analogously. 
	For $W \subseteq V (G)$, we define $d_G(v,W)$ as the number of edges between $v$ and~$W$.

	For graphs $G$ and $H$, $G-H$ is the subgraph of~$G$ obtained by deleting edges in $E(G) \cap E(H)$. 
	For $W \subseteq V (G)$, $G[W]$ and $G \setminus W$ are the induced subgraphs of~$G$ on~$W$ and~$V(G) \setminus W$, respectively.
	We write $G[A,B]$ for the induced bipartite subgraph of $G$ with parts $A$ and~$B$. 
	We often write $\{x, y\}$ as $xy$; for $xy \in E(G)$, let $T_G(xy)$ (or simply $T(xy)$ when $G$ is clear) denote the number of triangles in~$G$ that contain~$x$ and~$y$.

	Finally, let $[n]$ denote the set $\{1,\dots,n\}$ for $n\in \mathbb{N}$. Given a set $X$ and an integer $s$, we write $\binom{X}{s}$ for the set of all $s$-element subsets of $X$.
	We omit floors and ceilings whenever this does not affect the argument.

	\subsection{Organization of paper}
	In Section~\ref{Section: Proof of Theorem thmk3s}, we prove Theorem~\ref{thmk3s} and Proposition~\ref{2t^3 triangles}.
	In Section~\ref{Section: Constructions}, we give multiple constructions of $K_3(2)$-free tripartite graphs with minimum degree $n+c'n^{\frac{1}{2}}$ for some $c'\in (0,1)$. In Appendix~\ref{Sec:Appendix A}, we establish the result with minimum partial degree condition.

	\section{Proofs of Theorem~\ref{thmk3s} and Proposition~\ref{2t^3 triangles}}\label{Section: Proof of Theorem thmk3s}
	
	Let us denote by $z(m,n;s,s)$ the Zarankiewicz number, which is the maximum number of edges in a bipartite graph $G = (U, V ; E)$ with $|U| = m$ and $|V | = n$ containing no copy of $K_{s,s}$. The following well-known result on $z(m,n;s,s)$ was proved by K\"{o}v\'{a}ri, S\'{o}s and Tur\'{a}n~\cite{kHovari1954problem}.
	\begin{lemma}[\cite{kHovari1954problem}]\label{Zarankiewicz}
		For $s,m,n \in \mathbb{N}$, 
		$z(m,n;s,s)\le (s-1)^{\frac1s} mn^{1-\frac{1}{s}} + (s-1)n.$
	\end{lemma}

	Next we prove Theorem~\ref{thmk3s} by double counting and Lemma~\ref{Zarankiewicz}. This approach is similar to the one used in~\cite{bhalkikar2023subgraphs, bollobas1975complete} and our improvement on $\delta(G)$ is due to the counting of  triangles $xyz$ with $x\in T_1$ and $y\in T_x$ (where $T_1\subseteq V_1$ and $T_x\subseteq V_3$ are subsets of size $t$) instead of $x\in V_1$ and $y\in V_3$ as in~\cite{bhalkikar2023subgraphs, bollobas1975complete}.
	
	\begin{proof}[\textbf{Proof of Theorem~\ref{thmk3s}}]
		Let $t= \left\lceil Cn^{1-\frac{1}{s(s+1)}}\right\rceil$, where $C= 2 (s-1)^{\frac1{s+1}}$. Since $\delta(G)$ is an integer, the hypothesis of the theorem implies that $\delta(G)\ge n+t$. 
		Let $d_G^+(S)= \sum_{x\in S} d^+_G(x)$ for any set $S\subseteq V(G)$.
		Without loss of generality, suppose that there is a subset $T_1\subseteq V_1$ of size $t$ satisfying 
		\begin{align*}
			d_G^+(T_1) = \max \{ d_G^+(T) \colon \text{$|T|= t$, $T\subseteq V_i$, $i\in [3]$}  \}. 
		\end{align*}
		For $x\in V_1$, let $T_x\subset V_3$ be an arbitrary $t$-element subset of $N_G^-(x)$. 
		Note that $\delta(G)\ge n+t$ guarantees that $T_x$ exists. 
		Since $d^+(y) + d^-(y) \ge n+ t$ for any $y\in V(G)$, we have 
		\begin{align*}
			\sum\limits_{x\in T_1,y\in T_x} T (xy) & \geq \sum\limits_{x\in T_1}\sum\limits_{y\in T_x}\left(d_G^+(x)+d_G^-(y)-n\right)
			\geq \sum\limits_{x\in T_1}\sum\limits_{y\in T_x}\left(t+d_G^+(x)-d_G^+(y)\right)\\
			&= \sum\limits_{x\in T_1}\left(t^2+td_G^+(x)-  	d_G^+(T_x) \right)
			= t^3 + t d^+_G(T_1) -  \sum\limits_{x\in T_1} d_G^+(T_x) .
		\end{align*}
		Using the maximality of $d_G^+(T_1)$, we derive that		
		\begin{align*}
			\sum\limits_{x\in T_1,y\in T_x} T (xy) 
			\ge  t^3+t d_G^+(T_1) -t d_G^+(T_1)
			=t^3.
		\end{align*}
		For any $s$ distinct vertices $z_1,\dots,z_s\in V_2$, let 
		\begin{align*}
			\mathcal{T}(z_1,\dots,z_s) = \left\{ xy\in E(G): x\in T_1,y\in T_x, G[\{x, y, z_i\}] \text{ is a triangle for all } i\in [s] \right\},
		\end{align*}
        and let $q = \left\lfloor\frac{1}{t^2}{\sum\limits_{x\in T_1,y\in T_x}T(xy)}\right\rfloor$.
		By double counting and the discrete convexity of $\binom{k}{s}$ in its upper index $k$, we have
		\begin{align*}
			\sum\limits_{\{z_1,\dots,z_s\}\in\binom{V_2}{s}} |\mathcal{T}(z_1,\dots,z_s)|&= \sum\limits_{x\in T_1,y\in T_x}\binom{T(xy)}{s}
			\geq t^2\binom{q}{s}
			\geq t^2\binom{t}{s}.
		\end{align*}
		Since $s$ is fixed and $n$ is sufficiently large, we may assume that $n\ge s$ and $t\ge 2(s-1)$. By averaging, there exist $s$ distinct vertices $z_1,\dots,z_s\in V_2$ such that
		\begin{align*}
			|\mathcal{T}(z_1,\dots,z_s)|&\geq \frac{t^2\binom{t}{s}}{\binom{n}{s}}
			\geq \frac{t^{s+2}}{2^sn^s}
			\ge \frac{C^{s+1}}{2^s} t n^{1-\frac{1}{s}} =2(s-1)t n^{1-\frac1s}\\
			&> (s-1)t n^{1-\frac{1}{s}} + (s-1) n \ge  z(t, n; s, s).
					\end{align*}
			Here the strict inequality follows from $t>n^{\frac1s}$, which holds by the definition of $t$, since $C\ge2$ and $1-\frac{1}{s(s+1)}>\frac1s$. The last inequality follows from Lemma~\ref{Zarankiewicz} and the fact that
$(s-1)^{\frac1s}\le s-1$.
		Thus, the bipartite graph on $T_1\cup V_3$ with the edge set $\mathcal{T}(z_1,\dots,z_s)$ contains a copy of $K_{s, s}$. Together with $z_1,\dots,z_s$, this gives the desired copy of  $K_3(s)$ in $G$.		
	\end{proof}

We next prove Proposition~\ref{2t^3 triangles}.	
	
	\begin{proof}[\textbf{Proof of Proposition~\ref{2t^3 triangles}}]
		Let $G = G_3(n)$ with $\delta(G) \ge n+t$. 
		Let $\overline{G}$ be the tripartite complement graph of~$G$, that is $K_3(n) - G$, where $K_3(n)$ has the same vertex classes as~$G$. 
		Note that $\Delta(\overline{G}) \le n-t$ and $e(\overline{G}) \le 3n(n-t)/2$. 
		Since each edge is in at most $n$ triangles, so 
		\begin{align*}
			T(G) \ge T(K_3(n)) - n e(\overline{G}) 
			\geq n^3 - \frac{3n^2(n-t)}2
			= \frac12 n^2 (3t-n).
		\end{align*}
		Thus we have $f(n,t) \ge n^2 (3t-n)/2$.

		Suppose that $n$ is even and $t \ge n/2$. 
		Let $A_1, B_1, A_2, B_2, A_3,B_3$ be disjoint vertex sets each of size~$n/2$. 
		For $i \in [3]$, let $V_i = A_i \cup B_i$. 
		Let $H$ be a tripartite graph with vertex classes $V_1, V_2, V_3$ such that $H = \bigcup_{i \in [3]} H[A_{i+1},B_{i}]$ where each $H[A_{i+1},B_{i}]$ is $(n-t)$-regular. Indeed, since $K_{n/2,n/2}$ decomposes into $n/2$ edge-disjoint perfect matchings and $0\le n-t\le n/2$, each $H[A_{i+1},B_i]$ can be chosen as the union of $n-t$ of them.
		Let $G$ be the tripartite complement graph of~$H$, so $G$ is $(n+t)$-regular. 
		Since no triangle in $K_3(n)$ contains two edges of~$H$, by the calculation above we have $T(G) = n^2 (3t-n)/2$.
	\end{proof}

\section{New constructions of \texorpdfstring{$K_3(2)$}{}-free tripartite graphs}
\label{Section: Constructions}

For $n = q^2 + q + 1$ where $q$ is a prime power, it is well known (see~\cite{rieman1958problem}) that there is a $K_{2,2}$-free $(q+1)$-regular bipartite graph $G_0=G_2(n)$ (note that $q+1>\sqrt{n}$).
Using $G_0$ as a building block, Bhalkikar and Zhao~\cite{bhalkikar2023subgraphs} constructed many non-isomorphic $K_3(2)$-free tripartite graphs with minimum degree at least $n+n^{\frac{1}{2}}$. In this section, we present some new constructions that differ from those established in \cite{bhalkikar2023subgraphs}.

Our first construction is based on the one given by Bollob\'{a}s, Erd\H{o}s, and Szemer\'{e}di \cite{bollobas1975complete}.
\begin{construction}\label{construction 1}
	Given $n, t \in \mathbb{N}^+$ such that $n\geq \max\{t^2 + t+1, 5t\}$, we construct $G=G_3(n)$ as follows. Let $A_i\subseteq V_i$, $|A_i|=n-2t$, $B_i=V_i-A_i$, $i\in [3]$, and $B_1=\overline{B_2}\cup \overline{B_3}$, $|\overline{B_2}|=t=|\overline{B_3}|$. Join every vertex of $A_1$ to every vertex of $A_2 \cup A_3$, join every vertex of $\overline{B_j}$ to every vertex of $V_j$, $j = 2,3$, and join every vertex of $B_i$ to every vertex of $V_j$ for $i = 2, j = 3$ and $i = 3, j = 2$. Finally, join every vertex of $\overline{B_i}$ to $t$ vertices of $A_j$ for $i = 2, j = 3$ and $i = 3, j = 2$, such that any two different vertices of $\overline{B_i}$ have at most $1$ common neighbor in $A_j$.
\end{construction}
\begin{figure}[H]
	\centering
	\begin{tikzpicture}[scale=1.8]
	\def\x{0}
	\def\X{0}
	\def\l{0.6}
	\def\h{0.25}
	\def\skip{1.2}
	\def\y{1.2}
	\def\Y{-1.5}
	
	\node[circle, fill=black, inner sep=0.001] (BB2) at (\x,\X) {};
	\node at (\x,\X+\h) {$\overline{B_2}$};
	\draw (\x-\l/2,\X) rectangle (\x+\l/2,\X+\h*0.5);
	
	\node[circle, fill=black, inner sep=0.01] (A1) at (\x+\skip,\X+0.5*\h-\h) {};
	\node at (\x+\skip,\X+0.5*\h) {$A_1$};
	\draw (\x+\skip-\l*0.75,\X+0.5*\h-\h) rectangle (\x+\skip+\l*0.75,\X+0.5*\h+\h);
	
	\node[circle, fill=black, inner sep=0.001] (BB3) at (\x+2*\skip,\X) {};
	\node at (\x+2*\skip,\X+\h) {$\overline{B_3}$};
	\draw (\x+2*\skip-\l/2,\X) rectangle (\x+2*\skip+\l/2,\X+\h*0.5);
	
	
	\node[circle, fill=black, inner sep=0.001] (A2) at (\y-2,\Y+0.25*\h) {};
	\node at (\y-2,\Y-\h*0.5) {$A_2$};
	\draw (\y-2-\l*0.65,\Y-\h-0.25*\h) rectangle (\y-2+\l*0.65,\Y+0.25*\h);
	
	\node[circle, fill=black, inner sep=0.001] (B2) at (\y-2+\skip,\Y) {};
	\node[circle, fill=black, inner sep=0.001] (B2m) at (\y-2+\skip+\l/2,\Y-\h*0.5) {};
	\node at (\y-2+\skip,\Y-\h*0.5) {$B_2$};
	\draw (\y-2+\skip-\l/2,\Y-\h) rectangle (\y-2+\skip+\l/2,\Y);
	
	\node[circle, fill=black, inner sep=0.001] (A3) at (\y+2,\Y+0.25*\h) {};
	\node at (\y+2,\Y-\h*0.5) {$A_3$};
	\draw (\y+2-\l*0.65,\Y-\h-0.25*\h) rectangle (\y+2+\l*0.65,\Y+0.25*\h);
	
	\node[circle, fill=black, inner sep=0.001] (B3) at (\y+2-\skip,\Y) {};
	\node[circle, fill=black, inner sep=0.001] (B3m) at (\y+2-\skip-\l/2,\Y-\h*0.5) {};
	\node at (\y+2-\skip,\Y-\h*0.5) {$B_3$};
	\draw (\y+2-\skip-\l/2,\Y-\h) rectangle (\y+2-\skip+\l/2,\Y);
	
	\draw[black] (A2) -- (BB2) --(B2);
	\draw[black] (A2) -- (A1) --(A3);
	\draw[black] (A3) -- (BB3) --(B3);
	\draw[black] (A2) to[out=25, in=155] (B3m) --(B2m) to[out=25, in=155] (A3);
	
	\draw[red,thick,dashed,-] (BB2) -- (A3);
	\draw[red,thick,dashed,-] (BB3) -- (A2);
	
	\end{tikzpicture}
	\caption{Graph in Construction~\ref{construction 1}.}
	\label{figure:1.2}
\end{figure}
The condition $n\geq t^2 + t+1$ (which implies $t(t-1)\le n-2t-1$) ensures that we can join every vertex of $\overline{B_i}$ to $t$ vertices in $A_j$ for $i = 2, j = 3$ and $i = 3, j = 2$, with the additional constraint that any two distinct vertices of $\overline{B_i}$ share at most one common neighbor in $A_j$. To see this, for each $(i,j)\in\{(2,3),(3,2)\}$, fix a vertex $a_j\in A_j$. Since $|\overline{B_i}|=t$ and $|A_j|-1\ge t(t-1)$, for each $b\in\overline{B_i}$ we can choose a set $P_b\subseteq A_j\setminus\{a_j\}$ of size $t-1$ so that the sets $P_b$ are pairwise disjoint. Join each $b$ to $a_j$ and to every vertex of $P_b$. Then each $b\in\overline{B_i}$ has $t$ neighbors in $A_j$, and any two distinct vertices of $\overline{B_i}$ have $a_j$ as their only common neighbor in $A_j$. Furthermore, $n\geq 5t$ guarantees that each vertex in $A_1$ has degree at least $n+t$. Thus, in this construction, the minimum degree of $G$ is $n+t$, which can be chosen as $n+(1-o(1))n^{\frac{1}{2}}$ by setting $t = (1-o(1))n^{\frac{1}{2}}$. Moreover, $G$ is $K_3(2)$-free. Indeed, any copy of $K_3(2)$ would have to lie in either $G[\overline{B_2}\cup B_2\cup A_3]$ or $G[\overline{B_3}\cup B_3\cup A_2]$. It would then contain a copy of $K_{2,2}$ in $G[\overline{B_2},A_3]$ or $G[\overline{B_3},A_2]$, respectively. This contradicts the choice of the edges above.

In the following, we introduce a gluing operation to construct more examples.
\begin{construction}\label{gluing}
Let $G$ and $G'$ be two $K_3(2)$-free tripartite graphs on disjoint vertex classes $V_1, V_2,V_3$ and $V'_1, V'_2,V'_3$, respectively, where $|V_i|=|V'_i|=n$ for $i=1,2,3$. 
Define $G \odot G'$ to be the tripartite graph obtained from $G \cup G'$ by adding all edges between $V_i$ and~$V'_{i+1}$ for all $i \in [3]$. Then $\delta(G \odot G')= n+ \min \{ \delta(G), \delta(G')\}$.
Since no new edge in $G \odot G'$ lies in a triangle, $G \odot G'$ is $K_3(2)$-free.
\end{construction}

This construction allows us to create many new $K_3(2)$-free tripartite graphs. 
For example, let $G_1$ be the graph of Construction~\ref{construction 1} with $t = (1-o(1))n^{\frac{1}{2}}$.
Then $G = G_1 \odot G_1$ is a~$K_3(2)$-free tripartite graph with $2n$ vertices in each part and $\delta(G) = 2n+ (1-o(1))n^{\frac{1}{2}}$.
\medskip

Recall that Bollob\'{a}s, Erd\H{o}s, and Szemer\'{e}di~\cite{bollobas1975complete} conjectured there exists $c$ such that every $G = G_3(n)$ with $\delta(G) \ge n + cn^{1/2}$ contains $K_3(2)$. 
If $\delta^+(G) = \beta n$, then, after deleting the edges $uv$ with large $T(uv)$ and a small exceptional set of vertices as in Appendix~\ref{Sec:Appendix A}, Lemma~\ref{Lemma: 2 C_6 blow-ups} implies that $G$ essentially contains an almost blow-up of~$C_6$ with vertex classes of size about $\beta n$. 
After removing this almost blow-up of~$C_6$, we believe that one should be able to reduce to the case when both $\delta^+(G)$ and $\delta^-(G)$ are sublinear in~$n$.
However, we do not know how to handle this case.

	\section*{Acknowledgement}
	
	Jialin He was supported in part by Science and Technology Commission of Shanghai Municipality No. 22DZ2229014. Allan Lo was partially supported by EPSRC, grant no. EP/V002279/1 and EP/V048287/1. Jie Ma was supported by National Key R\&D Program of China 2023YFA1010201 and National Natural Science Foundation of China grant 12125106. Yi Zhao was partially supported by NSF grant DMS 2300346 and Simons Collaboration Grant 710094.
	There are no additional data beyond that contained within the main manuscript.
	For the purpose of open access, a CC BY public copyright license is applied to any AAM arising from this submission.

    \paragraph{Note added in proof.} After this paper was posted on arXiv as 2411.19773, Problem \ref{prob2} was settled by Di Braccio and Illingworth~\cite{DiBraccioIllingworth2024}, who showed that there is a constant $K$ such that every $G_3(n)$ with minimum degree at least $n + Kn^{1 - 1/s}$ contains a copy of $K_3(s)$. Their proof uses a density-increment argument, whereas our proof of Theorem~\ref{thmk3s} combines double counting with Zarankiewicz-type estimates. Beyond this, several other aspects of the present paper provide independent insights. Proposition~\ref{2t^3 triangles} addresses the related extremal problem of minimizing the number of triangles. The constructions in Section~\ref{Section: Constructions} exhibit additional, structurally different $K_3(2)$-free examples with minimum degree $n+(1-o(1))n^{1/2}$. Finally, Appendix~\ref{Sec:Appendix A} provides structural information through an approximate $C_6$-blow-up analysis under a minimum partial-degree condition.
	
	\bibliographystyle{unsrt}

	\vspace{0.5cm}
	
	\noindent{\it E-mail address}: cyh2020@mail.ustc.edu.cn
	\vspace{0.1cm}
	
	\noindent{\it E-mail address}: jlhe@math.ecnu.edu.cn
	\vspace{0.1cm}
	
	\noindent{\it E-mail address}: s.a.lo@bham.ac.uk
	\vspace{0.1cm}
	
	\noindent{\it E-mail address}: luoc@xtu.edu.cn
	\vspace{0.1cm}
	
	\noindent{\it E-mail address}: jiema@ustc.edu.cn
	\vspace{0.1cm}
	
	\noindent{\it E-mail address}: yzhao6@gsu.edu

\clearpage
\appendix
    \section{The result with minimum partial degree condition}\label{Sec:Appendix A}
In this appendix, we present and prove the result with a minimum partial degree condition.
    \begin{theorem}\label{theorem: partial degree}
		For every $c \ge 58$, there exists $n_0 = n_0(c)$ such that every tripartite graph $G=G_3(n)$ with $n \ge n_0$, $\delta(G)\geq n+30^5 c^4 n^{\frac{1}{2}}$, and the minimum partial degree at least $(1/5 +  7/c)n$ contains a~$K_3(2)$. 
	\end{theorem}

We first sketch the proof of Theorem~\ref{theorem: partial degree}.
	Suppose that $G$ is $K_3(2)$-free with minimum partial degree $\beta n\geq (1/5 + 7/c )n$. 
	Using the fact that $\delta^+(G) \ge \beta n$, we show that $G$ contains an almost blow-up of~$C_6$ with parts of size $\beta n +o(n)$ (Lemma~\ref{Lemma: 2 C_6 blow-ups}). 
	Moreover, for each vertex~$v$ in this almost $C_6$-blow-up, we have $d^+(v) \le \beta n +o(n)$ and $d^-(v) \ge n- \beta n +o(n)$.
	Similarly, by $\delta^-(G) \ge \beta n$, we obtain another almost $C_6$-blow-up with similar properties. 
	If these two almost $C_6$-blow-ups intersect, then it leads to a contradiction immediately. Otherwise, we use Lemma \ref{lma:C_6-close} to deduce that $G$ contains a~$K_3(2)$ and thus complete the proof.

	We begin with a definition. 
	Recall that $G=G_3(n)$ is viewed as an oriented graph with edges from $V_i$ to $V_{i+1}$ for $i\in [3]$. 
	For $v\in V(G)$ and $\alpha >0$, let 
	\begin{align*}
		\widetilde{D}^+_{ G,\alpha }(v)=\left \{ w \in N^+_G(v): T(vw) \geq \alpha n \right\}, \\ \widetilde{D}^-_{ G,\alpha }(v)=\left \{ w \in N^-_G(v): T(vw) \geq \alpha n \right\}.
	\end{align*}
	The following lemma shows that only a small number of vertices $w\in V(G)$ can have large $\widetilde{D}^+_{ G,\alpha }(w)$ or $\widetilde{D}^-_{ G,\alpha }(w)$.

	\begin{lemma}\label{proposition: upper bound for N^+G,alpha}
		Let $G=G_3(n)$ and $k > 1$.
		Suppose that for some $i\in [3],$ there exists a vertex subset $W_i\subseteq V_i$ of size $|W_i|\geq k^2n^{\frac12}$ such that either $|\widetilde{D}^+_{G,2/k}(w)|\geq k^2n^{\frac{1}{2}}$ for all $w\in W_i$ or $|\widetilde{D}^-_{G,2/k}(w)|\geq k^2n^{\frac{1}{2}}$ for all $w\in W_i$. Then $G$ contains a~$K_3(2)$.
	\end{lemma}
	\begin{proof}
		We only prove the $|\widetilde{D}^+_{G,2/k}(w)|\geq k^2n^{\frac{1}{2}}$ case because the proof of the $|\widetilde{D}^-_{G,2/k}(w)|\geq k^2n^{\frac{1}{2}}$ case is similar. Without loss of generality, we can assume $|W_i| = k^2 n^{\frac{1}{2}}$.
		For each $w\in W_i$, let $W(w)\subseteq \widetilde{D}^+_{G,2/k} (w)$ be of size~$k^2n^{\frac{1}{2}}$. 
		
		Define an auxiliary bipartite graph~$H$ with parts $W_i$ and $V_{i+1}\times V_{i+2}$ such that for $w\in W_i$ and $(v_{i+1},v_{i+2}) \in V_{i+1}\times V_{i+2}$, $w(v_{i+1},v_{i+2})$ is an edge of~$H$ if $v_{i+1}\in W(w)$ and $wv_{i+1}v_{i+2}$ forms a triangle in $G$.
			For all $w\in W_i$, $d_H(w) \geq k^2 n^{\frac{1}{2}} \cdot (2/k) n = 2 k n^{\frac{3}{2}}$, so $e(H)\geq 2k^{3}n^2$.
		Thus
		\begin{equation*}
			\begin{aligned}
				\sum\limits_{\{w_i,w_i'\}\in \binom{W_i}{2}}|N_H(w_i)\cap N_H(w_i')| &=  \sum\limits_{(v_{i+1},v_{i+2})\in V_{i+1}\times V_{i+2}}\binom{d_H\left((v_{i+1},v_{i+2})\right)}{2}\\
				& \geq |V_{i+1}\times V_{i+2}|\binom{e(H)/|V_{i+1}\times V_{i+2}|}{2}
				\geq n^2\binom{2k^3}{2}
				\geq k^6n^2.
			\end{aligned}
		\end{equation*}
		By averaging, there exists $\{w_i,w_i'\} \in \binom{W_i}{2}$ with $|N_H(w_i)\cap N_H(w_i')|> 2k^2n$.
		Note that $N_H(w_i)\cap N_H(w_i')$ can be viewed as a subgraph of $G[W(w_i),V_{i+2}]$.
		By Lemma~\ref{Zarankiewicz}, we have 
		\begin{align*}
			z\left(k^2 n^{\frac{1}{2}},n;2,2\right)\leq k^2n+n < |N_H(w_i)\cap N_H(w_i')|,
		\end{align*}
		implying $K_3(2)\subseteq G$.
	\end{proof}
	
	Now we introduce two lemmas and postpone their proofs to the next two subsections. The first lemma gives the structure of $K_3(2)$-free tripartite graph with $\delta(G)\geq n$ and linear partial degree. 
	
	\begin{lemma}\label{Lemma: 2 C_6 blow-ups}
		Let $\eps >0 $ and $n$ be sufficiently large.
		Suppose $G=G_3(n)$ is a tripartite graph with $\delta(G)\geq n$ and $\delta^+(G)\ge 2\eps n$. Further, assume that $T(uv)\le \left(\frac{\varepsilon}{30}\right)^2n$ for all edges $uv$ in~$G$. 
		Then either $G$ contains a~$K_3(2)$ or there exists 
		a partition $\mathcal{P}=\{W_1,\dots,W_6,U_1,U_2, U_3\}$ of $V(G)$ such that for $i\in [6]$ and $j\in [3]$,
		\begin{enumerate}[label = {\rm (\alph*)}]
			\item $W_j,W_{j+3},U_j \subseteq V_j$;
			\label{itm:C6:1}
			
			\item $\delta^+(G) -\eps n \leq |W_i| \le \delta^+(G) + \eps n$; \label{itm:C6:2a}	
			\item 		\label{itm:C6:3}		 for all $w\in W_i$, $S\in \{W_{i-1},W_{i+1},U_{i-1}\}$ and $S'\in \mathcal{P}\setminus \{W_{i-1},W_{i+1},U_{i-1}\}$ (the subscript of $W_i$ is modulo $6$ with values $1, \dots, 6$ while the subscript of $U_i$ is modulo $3$ with values $1,2, 3$), we have $d_G(w,S)\geq |S| - \eps n$, and $d_G(w,S') \le \eps n$;
			
			\item for all $w\in W_i$, $
			d_G^+(w) \leq \delta^+(G) + \eps n  $ and $ d_G^-(w) \ge n - \delta^+(G) - \eps n $.
			\label{itm:C6:4}
		\end{enumerate}
	\end{lemma}
	
The second lemma deals with the case when $G$ contains two almost $C_6$ blow-ups. 
	
	\begin{lemma}\label{lma:C_6-close}
Let $1<c\le n^{\frac{1}{6}}$ and $n\in \mathbb{N}$. 
Let $G=G_3(n)$ be a tripartite graph with $\delta(G)\geq n + 28 c^2 n^{\frac{1}{2}}$. Let $d=d(c,n)$ be a non-negative integer such that $\delta^+(G)\ge d$ and $3\delta^+(G)+2d\ge n+26c^{-1}n$. 
Suppose that there exist disjoint vertex sets $W_1,\dots,W_6,X_1,\dots,X_6$ (where $X_i$ can be empty for $i\in [6]$ and the subscripts of $X_i$ are modulo $6$) such that 
\begin{enumerate}[label = {\rm (\roman*)}]
	\item for all $j \in [3]$, $W_j, W_{j+3},X_{j},X_{j+3} \subseteq V_j$; \label{itm:c6close:2}
	\item for all $i \in [6]$, $\delta^+(G)-c^{-1}n \le |W_i| \le \delta^+(G)+c^{-1}n $ and $d-c^{-1}n \le |X_i| \le d+c^{-1}n $; \label{itm:c6close:1}
	\item 
	\label{itm:c6close:3}
	for all  $ i \in [6]$, $w \in W_i$ and $x \in X_i$, we have 
	\begin{align*}
	d(w,S) & \ge |S|- c^{-1} n & \text{if $S \in \{W_{i-1}, W_{i+1}, X_{i-1}, X_{i-4}\}$},\\
	d(x,S) & \ge |S|- c^{-1} n & \text{if $S \in \{X_{i-1}, X_{i+1}, W_{i+1}, W_{i+4}\}$}.
\end{align*}	
\end{enumerate}
Then $G$ contains a $K_{3}(2)$.
\end{lemma}
	\begin{figure}[H]
	\centering
	\begin{tikzpicture}[scale=1.5]
	\def\x{0}
	\def\X{0}
	\def\l{1}
	\def\h{3.3}
	\def\skip{1.5}
	\def\y{1.2}
	\def\Y{-1.5}
	
	\node at (\x,\X-0.2) {$V_1$};
	\draw (\x-\l/2,\X) rectangle (\x+\l/2,\X+\h);

 \node at (\x+\skip,\X-0.2) {$V_2$};
	\draw (\x-\l/2+\skip,\X) rectangle (\x+\l/2+\skip,\X+\h);

 \node at (\x+2*\skip,\X-0.2) {$V_3$};
	\draw (\x-\l/2+2*\skip,\X) rectangle (\x+\l/2+2*\skip,\X+\h);

\node[circle, fill=black, inner sep=0.001] (W1) at (\x+0.3,\X+5.3*\h/6) {};
\node at (\x,\X+5.3*\h/6) {$W_1$};
 \draw (\x,\X+5.3*\h/6) circle (0.3);

\node[circle, fill=black, inner sep=0.001] (W4) at (\x+0.3,\X+3.8*\h/6) {};
\node at (\x,\X+3.8*\h/6) {$W_4$};
 \draw (\x,\X+3.8*\h/6) circle (0.3);

\node[circle, fill=black, inner sep=0.001] (X1) at (\x+0.3,\X+2.3*\h/6) {};
\node at (\x,\X+2.3*\h/6) {$X_1$};
 \draw (\x,\X+2.3*\h/6) circle (0.3);

\node[circle, fill=black, inner sep=0.001] (X4) at (\x+0.3,\X+0.8*\h/6) {};
\node at (\x,\X+0.8*\h/6) {$X_4$};
 \draw (\x,\X+0.8*\h/6) circle (0.3);

\node[circle, fill=black, inner sep=0.001] (RW2) at (\x+0.3+\skip,\X+5.3*\h/6) {};
\node[circle, fill=black, inner sep=0.001] (LW2) at (\x-0.3+\skip,\X+5.3*\h/6) {};
\node at (\x+\skip,\X+5.3*\h/6) {$W_2$};
 \draw (\x+\skip,\X+5.3*\h/6) circle (0.3);

\node[circle, fill=black, inner sep=0.001] (RW5) at (\x+0.3+\skip,\X+3.8*\h/6) {};
\node[circle, fill=black, inner sep=0.001] (LW5) at (\x-0.3+\skip,\X+3.8*\h/6) {};
\node at (\x+\skip,\X+3.8*\h/6) {$W_5$};
 \draw (\x+\skip,\X+3.8*\h/6) circle (0.3);

\node[circle, fill=black, inner sep=0.001] (RX2) at (\x+0.3+\skip,\X+2.3*\h/6) {};
\node[circle, fill=black, inner sep=0.001] (LX2) at (\x-0.3+\skip,\X+2.3*\h/6) {};
\node at (\x+\skip,\X+2.3*\h/6) {$X_2$};
 \draw (\x+\skip,\X+2.3*\h/6) circle (0.3);

\node[circle, fill=black, inner sep=0.001] (RX5) at (\x+0.3+\skip,\X+0.8*\h/6) {};
\node[circle, fill=black, inner sep=0.001] (LX5) at (\x-0.3+\skip,\X+0.8*\h/6) {};
\node at (\x+\skip,\X+0.8*\h/6) {$X_5$};
 \draw (\x+\skip,\X+0.8*\h/6) circle (0.3);

\node[circle, fill=black, inner sep=0.001] (W3) at (\x-0.3+2*\skip,\X+5.3*\h/6) {};
\node at (\x+2*\skip,\X+5.3*\h/6) {$W_3$};
 \draw (\x+2*\skip,\X+5.3*\h/6) circle (0.3);

\node[circle, fill=black, inner sep=0.001] (W6) at (\x-0.3+2*\skip,\X+3.8*\h/6) {};
\node at (\x+2*\skip,\X+3.8*\h/6) {$W_6$};
 \draw (\x+2*\skip,\X+3.8*\h/6) circle (0.3);

\node[circle, fill=black, inner sep=0.001] (X3) at (\x-0.3+2*\skip,\X+2.3*\h/6) {};
\node at (\x+2*\skip,\X+2.3*\h/6) {$X_3$};
 \draw (\x+2*\skip,\X+2.3*\h/6) circle (0.3);

\node[circle, fill=black, inner sep=0.001] (X6) at (\x-0.3+2*\skip,\X+0.8*\h/6) {};
\node at (\x+2*\skip,\X+0.8*\h/6) {$X_6$};
 \draw (\x+2*\skip,\X+0.8*\h/6) circle (0.3);

 \draw[black] (W1) -- (LW2);
 \draw[black] (RW2) -- (W3);
 \draw[black] (W3) -- (W4);
 \draw[black] (W4) -- (LW5);
 \draw[black] (RW5) -- (W6);
 \draw[black] (W6) -- (W1);

 \draw[black] (X1) -- (LX2);
 \draw[black] (RX2) -- (X3);
 \draw[black] (X3) -- (X4);
 \draw[black] (X4) -- (LX5);
 \draw[black] (RX5) -- (X6);
 \draw[black] (X6) -- (X1);

	\end{tikzpicture}
	\caption{The two almost $C_6$ blow-ups in the graph of Lemma~\ref{lma:C_6-close}.}
	\label{figure:2 C_6 blow-ups}
\end{figure}
		
	\begin{proof}[\textbf{Proof of Theorem~\ref{theorem: partial degree}}]
		This proof of the theorem actually proves the following slightly stronger statement: \emph{For every $c \ge 58$, there exists $n_0 = n_0(c)$ such that every tripartite graph $G=G_3(n)$ with $n \ge n_0$, $\delta(G)\geq n+ 30^5 c^4 n^{\frac{1}{2}}$ and 
		\begin{align}\label{eq:32m}
	2\delta^+(G) + 2\delta^-(G) + \max\{\delta^+(G), \delta^-(G)\}\ge \left(1+35c^{-1}\right)n
		\end{align}
contains a~$K_3(2)$.} 

    Suppose to the contrary that there exists a~$K_3(2)$-free tripartite graph $G=G_3(n)$ satisfies the conditions above. Let $\alpha = (35c)^{-2}$ and $\tG$ be the spanning subgraph of $G$ with $E(\tG) = \{uv\in E(G): T(uv)\le \alpha n\}$. 
		For $i \in [3]$, let $S_i^+$ (and $S_i^-$) be a subset $S \subseteq V_i$ of size~$4 \alpha^{-2} n^{1/2}$ with $\sum_{v \in S} d^+_{G - \tG}(v)$ (and $\sum_{v \in S} d^-_{G - \tG}(v)$, respectively) maximal. 
By Lemma~\ref{proposition: upper bound for N^+G,alpha} with $k=2/\alpha$, we deduce that, for any $v\in V_i\setminus (S_i^+\cup S_i^-)$, we have 
        \begin{align}\label{eq:d+-}
        d^+_{\tG}(v)\ge d^+_G(v) - 4 \alpha^{-2} n^{1/2} \quad \text{and} \quad d^-_{\tG}(v)\ge d^-_G(v) - 4 \alpha^{-2} n^{1/2}.
        \end{align}
		
For each $i\in [3]$, we enlarge $S_i^+\cup S_i^-$, if necessary, to a set $\widehat S_i\subseteq V_i$ of size $8\alpha^{-2}n^{1/2}$. Let $G' = \tG \setminus \bigcup_{i\in [3]} \widehat S_i$.
Clearly, $G'$ is a tripartite graph with parts of size $n' = n- 8 \alpha^{-2}n^{1/2}$.
By~\eqref{eq:d+-} and the definition of $G'$, 
\begin{align*} 
        \delta^+(G')\ge \delta^+(\tG) - 8 \alpha^{-2}n^{1/2} > \delta^+(G) - 12 \alpha^{-2} n^{1/2}
\end{align*}
and analogously $\delta^-(G') \ge \delta^-(G) - 12 \alpha^{-2} n^{1/2}$. Using \eqref{eq:d+-} and $\alpha = (35c)^{-2}$, we obtain that 
        \begin{align*}
            \delta(G')&\ge \delta(G) - 24 \alpha^{-2} n^{\frac12} \ge n+ 30^5 c^4 n^{\frac{1}{2}} - {24}(35 c)^4 n^{\frac12} 
            \ge n'+ 28c^2n'^{\frac12}. 
        \end{align*}
        Without loss of generality, we assume $\delta^+(G')\ge \delta^-(G')$ (otherwise we reverse the direction of~$G'$). We claim that 
        \begin{align}\label{eq:32m'}
        3\delta^+(G')+2\delta^-(G')\ge n'+30c^{-1}n'.
        \end{align}
        Indeed, if $\delta^+(G)\ge \delta^-(G)$, then we have 
        \begin{align*}
            3\delta^+(G')+2\delta^-(G')&\ge 3\delta^+(G)+2\delta^-(G) - 60\alpha^{-2}n^{\frac12} \\
            &\overset{\mathclap{\text{\eqref{eq:32m}}}}{\ge} \left(1+35c^{-1}\right)n - 60\alpha^{-2}n^{\frac12} \ge \left(1+30c^{-1}\right)n' 
        \end{align*}
        as $n$ is sufficiently large. 
        If $\delta^-(G)\ge \delta^+(G)$, then, since $\delta^+(G')\ge \delta^-(G')$, we have       
         \begin{align*}
            3\delta^+(G')+2\delta^-(G')&\ge 3\delta^-(G')+2\delta^+(G')
            \ge 3\delta^-(G)+2\delta^+(G) - 60 \alpha^{-2} n^{\frac12} \\
          & \overset{\mathclap{\text{\eqref{eq:32m}}}}{\ge} \left(1+30c^{-1}\right)n'. 
        \end{align*}
        Hence \eqref{eq:32m'} holds.

        Now we prove that $G'$ contains a~$K_3(2)$, which contradicts our assumption that $G$ is $K_3(2)$-free. Let $V_1'$, $V_2'$, $V_3'$ be the three vertex classes of $G'$. Note that $T_{G'}(uv)\le T_{\tG}(uv)\le \alpha n = n/(35c)^2\le n'/(30c)^2$ for all $uv\in E(G')$.
				By~\eqref{eq:32m'}, $\delta^+(G') \ge n'/5$. 
        By Lemma~\ref{Lemma: 2 C_6 blow-ups} with $\varepsilon = c^{-1}$, $V(G')$ can be partitioned into $W_1,\dots,W_6,U_1,U_2,U_3$ such that, for $i\in [6]$, 
        \begin{enumerate}[label = {\rm (\alph*)}]
          \item $W_i,U_i \subseteq V'_{i\pmod{3}}$;\label{itm:C6:1:1}
          \item $\delta^+(G')-c^{-1}n'\le |W_i|\le \delta^+(G')+c^{-1}n'$;
          \item for all $w\in W_i$ and $S\in \{W_{i-1},W_{i+1},U_{i-1}\}$, $d_{G'}(w,S)\ge |S|-c^{-1}n'$; \label{itm:C6:4:3}
          \item for all $w\in W_i$, $d^+_{G'}(w)\le \delta^+(G')+c^{-1}n'$ and $d^-_{G'}(w)\ge n'-\delta^+(G')-c^{-1}n'$.\label{itm:C6:4:4}
        \end{enumerate}

        If $\delta^+(G')\ge (n'+26c^{-1}n')/{3}$, then we apply Lemma~\ref{lma:C_6-close} with $d=0$ and $X_i=\emptyset$ for all~$i\in [6]$ and obtain a~$K_3(2)$ in~$G'$, a contradiction.

        If $\delta^-(G')\le \delta^+(G')< (n'+26c^{-1}n')/{3}$, then by \eqref{eq:32m'}, we have
        \[
        \delta^-(G')\ge \frac12(n'+30c^{-1}n' - 3\delta^+(G'))  \ge \frac12 (n'+30c^{-1}n' - (n'+26c^{-1}n') )  \ge 2c^{-1}n'.
        \]       
        By reversing the direction of $G'$ and applying Lemma~\ref{Lemma: 2 C_6 blow-ups} with $\varepsilon = c^{-1}$ and noting $ \delta^-(G') \ge 2 c^{-1} n'$, we obtain a partition $W_1',\dots,W_6',U_1',U_2',U_3'$ of $V(G')$ such that, for $i\in [6]$, 
          \begin{enumerate}[label = {\rm (\alph*)}]
          \setcounter{enumi}{4}
          \item $W'_i,U'_i \subseteq V'_{5-i\pmod{3}}$;\label{itm:C6:1:5}
          \item $\delta^-(G')-c^{-1}n'\le |W'_i|\le \delta^-(G')+c^{-1}n'$;
          \item for all $w'\in W'_i$ and $S\in \{W'_{i-1},W'_{i+1},U'_{i-1}\}$, $d_{G'}(w',S)\ge |S|-c^{-1}n'$; \label{itm:C6:4:7}
          \item for all $w'\in W'_i$, $d^-_{G'}(w')\le \delta^-(G')+c^{-1}n'$ and $d^+_{G'}(w')\ge n'-\delta^-(G')-c^{-1}n'$.\label{itm:C6:4:8}
        \end{enumerate}      
        
        Let $X_i = W'_{8-i\pmod{6}}$ for $i\in [6]$. This and \ref{itm:C6:1:5} together imply that $X_i \subseteq V'_{i\pmod{3}}$.
         We claim that $W_1,\dots,W_6,X_1,\dots,X_6$ are pairwise disjoint. Indeed, suppose to the contrary that there exists a vertex $v\in \left(W_i\cup W_{i+3}\right)\cap (X_{i}\cup X_{i+3})$ for some $i\in [3]$. Then, by \ref{itm:C6:4:4} and \ref{itm:C6:4:8}, 
        we have $$n'-\delta^-(G')-c^{-1}n'\le d^+_{G'}(v)\le \delta^+(G')+c^{-1}n',$$ 
        which implies that $\delta^+(G')+\delta^-(G')\ge n'-2c^{-1}n'$, contradicting $\delta^-(G')\le \delta^+(G')< (n'+26c^{-1}n')/{3}$ as $c\ge 58$.

        Let $d=\delta^-(G')$. It is easy to see that all the assumptions of Lemma~\ref{lma:C_6-close} hold, for example,  \ref{itm:c6close:3} holds because of \ref{itm:C6:4:3} and \ref{itm:C6:4:7}, and the fact that $W_i\subseteq U'_{5-i\pmod{3}}$ and $W'_i\subseteq U_{5-i\pmod{3}}$ for $i\in [6]$.
        We can thus obtain a~$K_3(2)$ in $G'$ by applying Lemma~\ref{lma:C_6-close}. This contradicts our assumption and completes the proof of Theorem~\ref{theorem: partial degree}.
	\end{proof}

	\subsection{Proof of Lemma~\ref{Lemma: 2 C_6 blow-ups}}
	
	We begin with a simple proposition.
	\begin{proposition}\label{Proposition: A-B set}
		Let $0\leq \lambda\leq 1/{10}$ and $G$ be a bipartite graph with vertex classes $A$ and~$B$. Suppose that for all $a\in A$, $d_G(a)\geq (1-\lambda)|B|$. Then there exists a subset $B'\subseteq B$ of size $|B'|\geq (1-5\lambda)|B|$ such that for all $a\in A$ and $b\in B'$, $d_G(a,B')\geq (1-2\lambda)|B'|$ and $d_G(b,A)\geq 4|A|/5$.
	\end{proposition}

	\begin{proof}
		Let $B'=\{b\in B: d_G(b)\geq 4|A|/5\}$.
		Note that
		\begin{equation*}
			\begin{aligned}
				|A||B'|+\frac{4}{5}|A||B\setminus B'|&\geq \sum\limits_{b\in B}d_G(b)
				= e(G)
				=\sum\limits_{a\in A}d_G(a)
				\geq (1-\lambda)|B||A|.
			\end{aligned}
		\end{equation*}
		This implies $|B'|\geq (1-5\lambda)|B|\ge |B|/2$.
		Clearly $d_G(a,B')\geq |B'| - \lambda|B|\geq (1-2\lambda)|B'|$, and the result follows.
	\end{proof}

	\begin{proof}[\textbf{Proof of Lemma~\ref{Lemma: 2 C_6 blow-ups}}]
Suppose $G$ is $K_3(2)$-free. Let $\alpha = (\frac{\eps}{30})^2$. 
Without loss of generality, we will assume that there is an $a_0\in V_3$ with $d^+(a_0) = \delta^+(G) = \beta n $. 
Furthermore, for all $vw\in E(G)$ with $v\in V_{i-1}$ and $w\in V_{i}$, we have
\begin{align}
\alpha n &\ge T(vw) = | N(w)\cap N(v) | \ge d^+(w) + d^-(v) - n 
= d(w) - d^-(w) + d^-(v) - n \nonumber \\
& \ge \delta(G) - d^-(w) + d^-(v) - n 
\ge d^-(v) - d^-(w), \nonumber \\
			d^-(w)& \geq d^-(v)-\alpha n.	\label{eqn:vw}
	\end{align}

		Since we can view $G$ as an oriented graph with direction from $V_i$ to $V_{i+1}$ and $\delta^+(G) = \beta n$, we can find a directed path $P=a_0a_1\dots a_{12}$ of length $12$ in~$G$. 
		Let $A_0 = \{a_0\}$ and $A_i = {N}^+(a_{i-1})$ for $i\in [12]$. 
		Note that $a_i\in A_i$.

		\begin{claim}\label{Claim:A_i}
			The sets $A_i$'s satisfy the following properties.
			\begin{enumerate}[label = {\rm (\roman*)}]
				\item For $i\in [12]$, we have $|A_i|\geq \beta n$.
				\label{itm:C_6:1}
				
				\item For $i\in \{0,1,\dots,12\}$ and $v\in A_i$, we have $d^-(v)\geq (1-\beta-i\alpha)n$.
				\label{itm:C_6:2}
				
				\item For $i\in \{0,\dots,11\}$ and $v\in A_i$, we have $d(v, A_{i+1})\ge |A_{i+1}|-(i+1)\alpha n\geq (\beta -(i+1)\alpha)n$ and $d^+(v)\leq (\beta+i\alpha)n$. In particular, for $i\in [12]$, $|A_i|\leq d^+(a_{i-1})\leq (\beta +(i-1) \alpha)n$.
				\label{itm:C_6:3}
				
				\item For $i\in \{0,1,\dots,10\}$ and $v\in A_i$, we have $d^-(v,A_{i+2})\leq (i+3)\alpha n$.
				\label{itm:C_6:4}
				
				\item For $i\in \{0,1,\dots,9\}$, $|A_i\cap A_{i+3}|\leq 5\sqrt{\alpha} n$.
				\label{itm:C_6:5}
				
				\item For $i\in \{0,1,\dots,10\}$ and $v\in A_i$, we have $|V_{i-1}\setminus\left(A_{i+2}\cup N^-(v)\right)|\leq (2i+3)\alpha n$.
				\label{itm:C_6:6}
				
				\item For $i\in \{1,\dots,6\}$, $|A_i\cap A_{i+6}|\geq |A_i|-8\sqrt{\alpha} n\geq (\beta -8\sqrt{\alpha})n$.
				\label{itm:C_6:7}
			\end{enumerate}
		\end{claim}
		
		\begin{proofclaim}
			We have $|A_i| = d^+(a_{i-1})\ge \delta^+(G) = \beta n$ giving \ref{itm:C_6:1}.
			
			We prove \ref{itm:C_6:2} by induction on~$i$.
			For $i=0$, by using $\beta n = d^+(a_0)$, we have
			\begin{align*}
				d^-(a_0) = d(a_0) - d^+(a_0) \geq n - \beta n  = (1-\beta)n.
			\end{align*}
			So we may assume $i \in [12]$.
			For $v\in A_i$, note that $a_{i-1}v \in E(G)$.
			Together with~\eqref{eqn:vw}, we have 
			\begin{align*}
				d^-(v)\geq d^-(a_{i-1})-\alpha n\geq (1-\beta- (i-1)\alpha )n - \alpha n = (1-\beta-i\alpha) n.
			\end{align*}
			Hence \ref{itm:C_6:2} holds.
			
			For $i\in \{0,\dots,11\}$ and $v\in A_i$, we first show $d^+(v)\leq (\beta+i\alpha)n$ by using \ref{itm:C_6:2}. We know $d^+(a_0)= \beta n$. For $i\in [11]$, since $v\in A_i = N^+(a_{i-1})$, we have $|N^+(v)\cap N^-(a_{i-1})| = T( v a_{i-1}) \le \alpha n$. Consequently,
			\begin{align*}
				d^+(v) \le |N^+(v)\cap N^-(a_{i-1})| + |V_{i+1}\setminus N^-(a_{i-1})| 
				\overset{\mathclap{\text{\ref{itm:C_6:2}}}}{\le}
				\alpha n + (\beta + (i-1) \alpha)n = (\beta+i \alpha)n.
			\end{align*}
			This implies that $|A_i|\leq d^+(a_{i-1})\leq (\beta+(i-1)\alpha)n$ for $i\in [12]$.

			Next we show $d(v, A_{i+1})\ge |A_{i+1}|-(i+1)\alpha n\geq (\beta -(i+1)\alpha)n$ for $v\in A_i$ and $i\in \{0,\dots,11\}$. First, $d(a_0,A_1)=|A_1| > (\beta - \alpha) n$. 
			Note that 	
			\begin{align}
				\nonumber
				| V_{i+1}\setminus  ( N^-(a_{i-1}) \cup  N^+(v) )|
				& = |V_{i+1}\setminus N^-(a_{i-1})|  - |N^+(v) \setminus N^-(a_{i-1}))|\\
				\nonumber
				& \overset{\mathclap{\text{\ref{itm:C_6:2}}}}{\le }
				(\beta+ (i-1) \alpha)n - d^+(v) + T(a_{i-1} v)\\
				& \le (\beta+ (i-1) \alpha)n - \delta^+(G) + \alpha n 
				= i \alpha n.
				\label{eqn:C_6:1}
			\end{align}
It follows that 
			\begin{align*}
				d(v,A_{i+1})  
				& \ge |A_{i+1}| - |A_{i+1}\cap N^-(a_{i-1})| - |A_{i+1} \setminus  ( N^-(a_{i-1}) \cup  N^+(v) )|\\
				& \ge |A_{i+1}| - T(a_{i}a_{i-1}) - |V_{i+1} \setminus  ( N^-(a_{i-1}) \cup  N^+(v) )|\\
				& \overset{\mathclap{\text{\eqref{eqn:C_6:1}}}}{\ge }
				|A_{i+1}| - \alpha n - i \alpha n \overset{\mathclap{\text{\ref{itm:C_6:1}}}}{\ge } (\beta - (i+1)\alpha ) n.
			\end{align*}
			Hence \ref{itm:C_6:3} holds.
	
		Suppose $i\in \{0,1,\dots,10\}$ and $v\in A_i$. Let $w\in N^+(v,A_{i+1})$, which exists because \ref{itm:C_6:3} implies that $d^+(v, A_{i+1})\ge |A_{i+1}| - (i+1)\alpha n >0$. Then 
			\begin{align*}
				\alpha n \ge T(vw) &= |N^-(v)\cap N^+(w)|
				\geq |N^-(v)\cap N^+(w)\cap A_{i+2}| \\
				&\geq d^-(v,A_{i+2}) + d^+(w,A_{i+2})-|A_{i+2}| \\
				& \overset{\mathclap{\text{\ref{itm:C_6:3}}}}{ \ge } d^-(v,A_{i+2}) +(|A_{i+2}|-(i+2)\alpha n)-|A_{i+2}| = d^-(v,A_{i+2}) -(i+2)\alpha n.
			\end{align*}			
			This gives $d^-(v,A_{i+2})\le (i+3)\alpha n$, confirming 
			\ref{itm:C_6:4}. 			
			
			By~\ref{itm:C_6:1} and~\ref{itm:C_6:3}, for all $v\in A_{i+2}$, we have 
			\begin{align*}
				d^+(v,A_{i+3})>|A_{i+3}|-(i+3)\alpha n \geq (1-\sqrt{\alpha})|A_{i+3}|,
			\end{align*} 
			here we use the assumption that $\alpha = (\eps/ 30)^2$ and $\beta \ge 30\sqrt{\alpha}$. 
			By applying Proposition~\ref{Proposition: A-B set} on $G[A_{i+2},A_{i+3}]$, we have a subset $A_{i+3}'\subseteq A_{i+3}$ with size $(1-5\sqrt{\alpha})|A_{i+3}|$ such that, for all~$w\in A_{i+3}'$, 
			\begin{align*}
				d(w,A_{i+2})\geq \frac{4}{5}|A_{i+2}| > (i+3)\alpha n.
			\end{align*}
			Since $d(v, A_{i+2})\le (i+3)\alpha n$ for all $v\in A_i$ by~\ref{itm:C_6:4}, it follows that $A_i\cap A_{i+3}\subseteq A_{i+3}\setminus A_{i+3}'$. Therefore, $|A_i\cap A_{i+3}| \le 5\sqrt{\alpha} |A_{i+3}|\le 5\sqrt{\alpha} n$ confirming \ref{itm:C_6:5}.
			
			For all $v\in A_i$,
			\begin{align*}
				d^-(v,V_{i-1}\setminus A_{i+2})& = d^-(v)-d^-(v,A_{i+2})
				\overset{\mathclap{\text{\ref{itm:C_6:2},\ref{itm:C_6:4}}}}{ \ge }
				(1-\beta-i\alpha )n-(i+3) \alpha n\\
				&= (1-\beta-(2i+3) \alpha )n
				\overset{\mathclap{\text{\ref{itm:C_6:1}}}}{ \ge }
				|V_{i-1}|-|A_{i+2}|-(2i+3)\alpha n\\
				& = |V_{i-1}\setminus A_{i+2}|-(2i+3)\alpha n.
			\end{align*}
			Hence $|V_{i-1}\setminus\left(A_{i+2}\cup N^-(v)\right)|\leq (2i+3)\alpha n$ and \ref{itm:C_6:6} holds.
			
			Finally, for $0\le i\le 6$ and $v\in A_i$, we have
			\begin{align*}
				d^-(v,A_{i+5})& \geq  d^-(v,A_{i+5}\setminus A_{i+2})
				\geq |A_{i+5}\setminus A_{i+2}|-|V_{i-1}\setminus\left(A_{i+2}\cup N^-(v)\right)|\\
				&			\overset{\mathclap{\text{\ref{itm:C_6:5},\ref{itm:C_6:6}}}}{ \ge }
				|A_{i+5}|-5\sqrt{\alpha} n-(2i+3)\alpha n
				\geq  |A_{i+5}|-6\sqrt{\alpha } n.
			\end{align*}
			Hence there exists a vertex $w\in A_{i+5}$ such that
			\begin{align*}
				d^+(w,A_i)&\geq \frac{e(A_{i},A_{i+5})}{|A_{i+5}|} \ge 		
				\frac{|A_i|(|A_{i+5}|-6\sqrt{\alpha} n)}{|A_{i+5}|}
				=|A_i|-\frac{|A_i|}{|A_{i+5}|}6\sqrt{\alpha} n
				\geq |A_i|-7\sqrt{\alpha} n,
			\end{align*}
			where the last inequality holds because $|A_{i+5}|\geq \beta n$ and $|A_i|\leq (\beta + (i-1)\alpha)n$ by~\ref{itm:C_6:1} and~\ref{itm:C_6:3}, and consequently, $\frac{|A_i|}{|A_{i+1}|}\le \frac{\beta + 5\alpha}{\beta}<\frac76$ by our assumption on $\alpha$ and $\beta$. 
			Therefore,
			\begin{align*}
				|A_i\cap A_{i+6}|&\geq |A_i\cap A_{i+6}\cap N^+(w)|
				\geq d^+(w,A_i)+d^+(w,A_{i+6})-d^+(w)\\
				&
				\overset{\mathclap{\text{\ref{itm:C_6:3}}}}{>}
				|A_i|-7\sqrt{\alpha} n + (\beta -(i+6)\alpha)n -(\beta+(i+5)\alpha)n
				\geq |A_i|-8\sqrt{\alpha} n,
			\end{align*}
			confirming \ref{itm:C_6:7}.
		\end{proofclaim}
		
		Now we come back to the proof of the lemma. For $i\in [6]$, let $W_i=\left(A_i\cap A_{i+6}\right)\setminus A_{i+3}$. If there exists some $i\in [3]$ such that $|W_i|+|W_{i+3}|>n$, then this contradicts the fact that $W_i,W_{i+3}\subseteq V_i$ and implies that $G$ contains a~$K_3(2)$. Otherwise, for $j\in [3]$, let $U_j=V_j\setminus \left(W_j\cup W_{j+3}\right)$. Since $W_{j+3}\subseteq A_{j+3}$ and $W_j\cap A_{j+3}=\emptyset$, it follows that $W_j, W_{j+3}$, and $U_j$ are pairwise disjoint subsets of $V_j$, in particular, \ref{itm:C6:1} holds.
Hence $\mathcal{P}=\{W_1,\dots,W_6,U_1,U_2, U_3\}$ is a partition of $V(G)$.

		By Claim~\ref{Claim:A_i}~\ref{itm:C_6:3}, \ref{itm:C_6:5} and~\ref{itm:C_6:7}, for $i\in [6]$, we have 
		\begin{align}\label{eq:Wi}
               \nonumber
			(\beta -13\sqrt{\alpha})n &\leq |A_i\cap A_{i+6}| - |A_i\cap A_{i+3}|\\
			&\le |W_i| \le |A_{i}| \leq (\beta+5\alpha)n.
		\end{align}
		Consequently, 
			$(1-2\beta-10 \alpha)n\leq |U_j|\leq (1-2\beta+26\sqrt{\alpha})n$.
			Since $\eps = 30 \sqrt{\alpha}$, \ref{itm:C6:2a} holds.

		Consider $i \in [6]$ and $v \in W_i$. 
		Trivially, $d(v,W_{i+3})=0=d(v,U_i)$ by~\ref{itm:C6:1}. By Claim~\ref{Claim:A_i}~\ref{itm:C_6:5} and~\ref{itm:C_6:7}, we have $|W_{i+1}|\geq |A_{i+1}|-13\sqrt{\alpha}n$ and thus
		\begin{align}
			d(v,W_{i+1})& \geq d(v,A_{i+1})-13\sqrt{\alpha}n
			\overset{\mathclap{\text{\ref{itm:C_6:3}}}}{>}
			|A_{i+1}|-(i+1)\alpha n -13\sqrt{\alpha}n \nonumber \\		
			&\ge |W_{i+1}|- 7\alpha n -13\sqrt{\alpha}n \ge 
			(\beta - 27\sqrt{\alpha})n. \label{eqn:vW}
		\end{align}
		Together with Claim~\ref{Claim:A_i}~\ref{itm:C_6:3}, this implies that
		\begin{align*}
			d(v,W_{i+4}) +d_G(v,U_{i+1}) & = d^+(v)-d(v,W_{i+1})
			\le (\beta+i\alpha )n - (\beta - 27\sqrt{\alpha})n
			\le 28\sqrt{\alpha}n .
		\end{align*}
		By Claim~\ref{Claim:A_i}~\ref{itm:C_6:4}, 
		\begin{align*}
			d(v,W_{i+2}) \le d^-(v,A_{i+2}) \le (i+3) \alpha n \le 9\alpha n.
		\end{align*}
		Together with Claim~\ref{Claim:A_i}~\ref{itm:C_6:2}, this implies that
		\begin{align*}
			d_G(v,W_{i-1}) + d_G(v,U_{i-1})& = d_G^-(v)-d_G(v,W_{i+2})
			\ge (1-\beta-i\alpha)n - 9 \alpha n \\
			&\geq (1-\beta -15 \alpha ) n \overset{\mathclap{\text{\eqref{eq:Wi}}}}{\ge }  n- |W_{i+2}|-14\sqrt{\alpha} n \\
			&= |W_{i-1}| + |U_{i-1}| - 14\sqrt{\alpha}n.
		\end{align*}
		Therefore, \ref{itm:C6:3} holds. Finally, since $v\in W_i\subseteq A_i$, we have $d^+(v)\le (\beta + i\alpha) n\le (\beta + 6\alpha) n$ by Claim~\ref{Claim:A_i}~\ref{itm:C_6:3}, and $d^-(v)\ge (1-\beta- i\alpha)n\ge (1-\beta- 6\alpha)n$ by Claim~\ref{Claim:A_i}~\ref{itm:C_6:2}. Thus \ref{itm:C6:4} holds. This completes the proof of the lemma. 
	\end{proof}
	
\subsection{Proof of Lemma~\ref{lma:C_6-close}}
\begin{proof}[\textbf{Proof of Lemma~\ref{lma:C_6-close}.}]
For $i\in [3]$, let $R_i=V_i\setminus \left(W_i\cup W_{i+3}\cup X_{i} \cup X_{i+3}\right)$. By \ref{itm:c6close:1}, we have $|R_i|\le n-2\delta^+(G)-2d+4c^{-1}n$. Suppose to the contrary that $G$ is $K_3(2)$-free. We first prove the following claim.

\begin{claim}
There exist disjoint vertex subsets $W^*_1,\dots,W^*_6, X^*_1,\dots ,X^*_6$ such that 
\begin{enumerate}[label = {\rm (\roman*$'$)}]
	\item for all $i \in [6]$, $W_{i} \subseteq W_{i}^* \subseteq V_i$, $X_{i}\subseteq {X_{i}^*} \subseteq V_i$;  \label{itm:c6close:1'}
	\item for all $v \in W^*_i$, $ i \in [6]$ and $j \in \{i+1,i-1\}$, $d(v,W_{j})  \ge 3c^{-1} n $. For all $v\in X_{i}^*$, $ i \in [6]$ and $j \in \{i+1,i-2\}$, $d(v,W_{j})  \ge 3c^{-1} n $;\label{itm:c6close:4}
	\item $| V(G) \setminus \bigcup_{i \in [6]}\left(W^*_i \cup X_{i}^*\right) | \le 24c^2 n^{\frac12}$.  \label{itm:c6close:5}
\end{enumerate}
\end{claim}

\begin{proofclaim}
Let $\V= \bigcup_{i \in [6]}\{W_i, X_i\}$. 
For all $v\in R_1\cup R_2 \cup R_3$, we define $$I_v=\left\{A\in \V: d(v, A)\ge 3c^{-1}n \right\}.$$ 
Let \[
\mathcal{E}= \bigcup_{i \in [6]}\left\{ \{W_i, W_{i+1}\}, \{X_i, X_{i+1}\}, \{W_{i-2}, X_i\}, \{W_{i+1}, X_i\}  \right\}
\] 
be a family of $24$ pairs of $\V$.
For every $\{A, B\}\in \mathcal{E}$, we claim that at most $c^2 n^{\frac{1}{2}}$ vertices $v\in R_1\cup R_2 \cup R_3$ satisfy $\{A, B\}\subseteq I_v$ and call such vertices \emph{bad}. 
Indeed, for $i\in [6]$, let $Y_i$ be the set of vertices $v\in R_1\cup R_2 \cup R_3$ such that $\{W_i, W_{i+1}\}\subseteq I_v$  
(the argument for other pairs of $\mathcal{E}$ is similar). Note that $Y_i \subseteq V_{i+2}$. For all $v\in Y_{i}$, by the definition of $I_v$, we have $d(v,W_i)\ge 3c^{-1}n$ and $d(v,W_{i+1})\ge 3c^{-1}n$. Then at least $3c^{-1}n$ vertices $ w \in N_G(v, W_{i}) \subseteq  N^+_G(v)$ satisfy 
\begin{align*}
	T(vw) \ge | N_G(v,W_{i+1}) \cap N_G(w,W_{i+1}) | 
	\overset{\mathclap{\text{\ref{itm:c6close:3}}}}
	\ge 3c^{-1} n  - c^{-1} n = 2c^{-1}n,
\end{align*}
implying that $|\widetilde{D}_{G,2c^{-1}}^{+}(v)| \ge 3 c^{-1} n$. 
By Lemma~\ref{proposition: upper bound for N^+G,alpha}, we have $|Y_i|\le c^2 n^{\frac12}$, as claimed.

Now we delete all bad vertices and denote the remaining set by $R_1'\cup R_2' \cup R_3'$ (and call their vertices \emph{good} vertices). For $i\in [6]$, define $$W_i^* = W_i\cup \left\{v\in R_1'\cup R_2' \cup R_3': \{W_{i-1},W_{i+1}\}\subseteq I_v\right\}$$ and $$X_{i}^* =\begin{cases}
      X_{i}\cup \left\{v\in R_1'\cup R_2' \cup R_3': \{ W_{i+1},W_{i-2}\}\subseteq I_v\right\} , & i\in \{1,2,3\};\\
      X_i, & i\in \{4,5,6\}.
     \end{cases} $$ 
As before, the subscripts of $W_i^*$ and $X_i^*$ in these definitions are modulo~$6$ with values $1, \dots, 6$. We first show that $W_i^*,X_i^*\subseteq V_i$ for every $i\in[6]$. For $v\in W_i^*\setminus W_i$, by the definition of $W_i^*$, we have $\{W_{i-1},W_{i+1}\}\subseteq I_v$. Since $W_{i-1}\subseteq V_{i-1}$ and $W_{i+1}\subseteq V_{i+1}$, we have $v\in V_i$. Hence $W_i^*\subseteq V_i$. Also, by the definition, we have $X_i^*=X_i\subseteq V_i$ for $i\in\{4,5,6\}$. Let $i\in[3]$ and $v\in X_i^*\setminus X_i$. Then $\{W_{i+1},W_{i-2}\}\subseteq I_v$, so $v\notin V_{i+1}$. If $v\in V_{i-1}$, then the definition of $\mathcal E$ and the fact that $v$ is good give that $\{W_i,W_{i+3},X_i,X_{i+3}\}\cap I_v=\emptyset.$
Consequently, by the definition of $R_i$, we have
$$
 \delta^+(G)\le d^+(v)=d(v,V_i)
 <4\cdot 3c^{-1}n+|R_i|
 \le n-2\delta^+(G)-2d+16c^{-1}n.
$$
This contradicts $3\delta^+(G)+2d\ge n+26c^{-1}n$. Hence $X_i^*\subseteq V_i$ for every $i\in[6]$. Next we prove that $W_1^*,\dots,W_6^*,X_1^*,\dots,X_6^*$ are pairwise disjoint. By the assumption, the sets $W_1,\dots,W_6,X_1,\dots,X_6$ are pairwise disjoint. Moreover, $X_{i+3}^*=X_{i+3}$ for $i\in[3]$, and $R_1'\cup R_2'\cup R_3'$ meets none of the sets $W_j$ and $X_j$ for $j\in[6]$. Hence the only possible intersections are
$$
 W_i^*\cap W_{i+3}^*,
 W_i^*\cap X_i^* \text{ and }
 W_{i+3}^*\cap X_i^*,
$$
where $i\in[3]$. Any vertex $v$ in one of these intersections lies in $R_1'\cup R_2'\cup R_3'$ and is therefore good. $v\in W_i^*\cap W_{i+3}^*$ and $v\in W_{i+3}^*\cap X_i^*$ give that $\{W_{i+1},W_{i+2}\}\subseteq I_v$, while $v\in W_i^*\cap X_i^*$ gives that $\{W_{i-2},W_{i-1}\}\subseteq I_v$. Both pairs belong to $\mathcal E$, a contradiction. Therefore $W_1^*,\dots,W_6^*,X_1^*,\dots,X_6^*$ are pairwise disjoint.

Finally we show that every vertex in $R_1'\cup R_2' \cup R_3'$ belongs to some $W_i^*$ or $X_{i}^*$, which completes the proof. Suppose to the contrary that there is a vertex $v\in R_1'\cup R_2' \cup R_3'$ which does not belong to any $W_i^*$ and $X_{i}^*$ for $i\in [6]$. We will show that $|I_v\cap \{W_i: i\in [6]\} |\le 1$ and $|I_v\cap \{X_i: i\in [6] \}|\le 2$. By \ref{itm:c6close:1}, 
it follows that
\begin{align*}
	\delta(G)\le d(v)&< \delta^+(G) +c^{-1}n+2(d+c^{-1}n)+5\cdot 3c^{-1}n+2\left(n-2\delta^+(G)-2d+4c^{-1}n\right)\\ &= 2n+26c^{-1}n- \left(3\delta^+(G) +2d\right)\le n,
\end{align*}
contradicting our assumption.

Without loss of generality, assume that $v\in R_2'\subseteq V_2$. Trivially $I_v\subseteq\{W_1,W_3,W_4, W_6$, $X_1,X_3,X_4, X_6\}$. Since $v$ is a good vertex, if $|I_v\cap \{W_i: i\in [6]\} |\ge 2$, then $I_v$ contains either 
$\{W_1,W_4\}$ or $\{W_3, W_6\}$ or $\{W_1,W_3\}$ or $\{W_4,W_6\}$.
If $\{W_1,W_4\}\subseteq I_v$, then by the definition of good vertices, we have $\{W_3, W_6$, $X_3, X_6\}\cap I_v= \emptyset$, and thus $$\delta^+(G)\le d^+(v)\le 4\cdot 3c^{-1}n+|R_3|\le n-2\delta^+(G)-2d+16c^{-1}n,$$ implying $3\delta^+(G)+2d<n+16c^{-1}n$, a contradiction. If $\{W_3, W_6\}$, $\{W_1,W_3\}$ or $\{W_4,W_6\} \subseteq I_v$, then $v$ belongs to $X_{2}^*$, $W_2^*$ or $W_{5}^*$ respectively, contradicting our assumption on $v$. Thus $|I_v\cap \{W_i: i\in [6]\} |\le 1$. 
Furthermore, since $\{X_3, X_4\}\not\subseteq I_v$ and $\{X_1, X_6\}\not\subseteq I_v$, we have $|I_v\cap \{X_i: i\in [6] \}|\le 2$, as claimed.
\end{proofclaim}

Now we go back to the proof of the lemma. Firstly, we have $$\sum\limits_{i\in [6]}|W_{i+1}^*\cup W_{i-1}^*\cup X_{i+2}^* \cup X_{i-1}^*|=2\left(\sum\limits_{i\in [6]}\left(|W_i^*|+|X_{i}^*|\right)\right)\le 6n.$$ Hence there exists an $i\in [6]$ such that $|W_{i+1}^*\cup W_{i-1}^*\cup X_{i+2}^* \cup X_{i-1}^*|\le n$. Without loss of generality, we assume $|W_{2}^*\cup W_{6}^*\cup X_{3}^* \cup X_{6}^*|\le n$, and thus we have $$\delta(G)\overset{\mathclap{\text{\ref{itm:c6close:5}}}}{\ge} |W_{2}^*\cup W_{6}^*\cup X_{3}^* \cup X_{6}^*|+|V(G)\setminus\bigcup\limits_{i\in [6]}(W_i^*\cup X_i^*)|+4c^2n^{\frac12}.$$ So for all $v\in W_1\subseteq V_1$, we have $\left|N_G(v)\cap \left(W_{3}^*\cup W_{5}^*\cup X_{2}^* \cup X_{5}^*\right)\right|\ge 4c^2n^{\frac12}$. Therefore, there must exist a set $A\in \{W_{3}^*,W_{5}^*,X_{2}^*,X_{5}^*\}$ and a subset $\widetilde{W_1}\subseteq W_1$ with $|\widetilde{W_1}|\ge \frac{|W_1|}{4}\ge c^2n^{\frac12} $ such that for all $w\in \widetilde{W_1}$, $d(w,A)\ge c^2n^{\frac{1}{2}}$. If $A$ is one of $W_{5}^*$, $X_{2}^* $ and $ X_{5}^*$, then for all $w'\in N_G(w,A)$ we have $d(w',W_6)\ge 3c^{-1}n$ by \ref{itm:c6close:4}, which implies that 
$$T(ww')\ge |N_G(w',W_6)\cap N_G(w,W_6)|\overset{\mathclap{\text{\ref{itm:c6close:3}}}}\ge 3c^{-1} n  - c^{-1} n = 2c^{-1}n.$$ 
Thus for all $w\in \widetilde{W_1}$, we have $|\widetilde{D}_{G,2c^{-1}}^{+}(w)| \ge |N_G(w,A)| \ge c^{2}n^{\frac12}$. 
By Lemma~\ref{proposition: upper bound for N^+G,alpha}, $G$ contains a~$K_3(2)$, a contradiction. If $A=W_3^*$, then for all $w'\in N_G(w,A)$ we have $d(w',W_2)\ge 3c^{-1}n$. A similar argument shows that $G$ contains a~$K_3(2)$, this contradicts our assumption and completes the proof of Lemma~\ref{lma:C_6-close}.
\end{proof}
    

\begin{thebibliography}{99}
		
		\bibitem{alon1988linear}
		N. Alon,
		\newblock{The linear arboricity of graphs},
		\newblock{\emph{Isr. J. Math.}} \textbf{62(3)} (1988), 311--325.

		\bibitem{bhalkikar2023subgraphs}
		A. Bhalkikar and Y. Zhao,
		\newblock{On subgraphs of tripartite graphs},
		\newblock{\emph{Discrete Math.}}
		\textbf{346(1)} (2023), 113152.
		
		
		
		\bibitem{bollobas1974complete}
		B. Bollob{\'a}s, P. Erd\H{o}s and E. G. Straus, 
		\newblock{Complete subgraphs of chromatic graphs and hypergraphs},
		\newblock{\emph{Utilitas Math.}} \textbf{6} (1974), 343--347.
		
		\bibitem{bollobas1975complete}
		B. Bollob{\'a}s, P. Erd\H{o}s and E. Szemer{\'e}di,
		\newblock{On complete subgraphs of $r$-chromatic graphs},
		\newblock{\emph{Discrete Math.}}
		\textbf{13(2)} (1975), 97--107.

        \bibitem{DiBraccioIllingworth2024}
		F. Di Braccio and F. Illingworth,
		\newblock{The Zarankiewicz problem on tripartite graphs}, 
		\newblock{\emph{J. Lond. Math. Soc.}}
		\textbf{114(1)} (2026), e70621.
		

        \bibitem{glock2022average}
		S. Glock and B. Sudakov,
		\newblock{An average degree condition for independent transversals},
		\newblock{\emph{J. Comb. Theory, B}} \textbf{154} (2022), 370--391.

        \bibitem{Graf2020Finding}
		A. Graf and P. Haxell,
		\newblock{Finding independent transversals efficiently},
		\newblock{\emph{Combin. Probab. Comput.}} \textbf{29(5)} (2020), 780--806.

        \bibitem{han2013On}
		J. Han and Y. Zhao,
		\newblock{On multipartite Hajnal--Szemer{\'e}di theorems},
		\newblock{\emph{Discrete Math.}} \textbf{313(10)} (2013), 1119--1129.
        
		\bibitem{haxell2001note}
		P. Haxell,
		\newblock{A note on vertex list colouring},
		\newblock{\emph{Combin. Probab. Comput.}} \textbf{10(4)} (2001), 345--347.
		
		\bibitem{haxell2006odd}
		P. Haxell and T. Szab{\'o}, 
		\newblock{Odd independent transversals are odd},
		\newblock{\emph{Combin. Probab. Comput.}} \textbf{15(1-2)} (2006), 193--211.

        \bibitem{haxell2024degree}
		P. Haxell and R. Wdowinski, 
		\newblock{Degree criteria and stability for independent transversals},
		\newblock{\emph{J. Graph Theory.}} \textbf{106(2)} (2024), 352--371.
		
		\bibitem{jin1992complete}
		G. Jin,
		\newblock{Complete subgraphs of $r$-partite graphs},
		\newblock{\emph{Combin. Probab. Comput.}} \textbf{1(3)} (1992), 241--250.

        \bibitem{kang2022colorings}
		R.~J.~Kang and T.~Kelly,
		\newblock{Colorings, transversals, and local sparsity},
		\newblock{\emph{Random Struct. Algorithms.}} \textbf{61(1)} (2022), 173--192.
		
		\bibitem{kHovari1954problem}
		T. K{\H{o}}v{\'a}ri, V. T. S{\'o}s and P. Tur{\'a}n,
		\newblock{On a problem of Zarankiewicz},
		\newblock{\emph{Colloq. Math.}} \textbf{3} (1954), 50--57.

        \bibitem{lo2013Amulti}
		A. Lo and K. Markstr{\"o}m,
		\newblock{A multipartite version of the Hajnal--Szemer{\'e}di theorem for graphs and hypergraphs},
		\newblock{\emph{Combin. Probab. Comput.}} \textbf{22(1)} (2013), 97--111.
		
		\bibitem{lo2022complete}
		A. Lo, A. Treglown and Y. Zhao,
		\newblock{Complete subgraphs in a multipartite graph},
		\newblock{\emph{Combin. Probab. Comput.}} \textbf{31(6)} (2022), 1092--1101.

        \bibitem{loh2007Independent}
		P. S. Loh and B. Sudakov,
		\newblock{Independent transversals in locally sparse graphs},
		\newblock{\emph{J. Comb. Theory, B}} \textbf{97(6)} (2007), 904--918.
		
		\bibitem{mantel1907opgaven28}
		W. Mantel,
		\newblock{Opgaven28},
		\newblock{\emph{Wiskd. Opgaven Met Oplossingen}} \textbf{10} (1907), 60--61.
		
		
		\bibitem{rieman1958problem}
		I. Reiman,
		\newblock{{\"U}ber ein Problem von K. Zarankiewicz},
		\newblock{\emph{Acta Math. Acad. Sci. Hungar.}} \textbf{9} (1958), 269--273.
		
		\bibitem{szabo2006extremal}
		T. Szab{\'o} and G. Tardos, 
		\newblock{Extremal problems for transversals in graphs with bounded degree},
		\newblock{\emph{Combinatorica}} \textbf{26(3)} (2006), 333--351.
		
		\bibitem{turan1941external}
		P. Tur{\'a}n,
		\newblock{On an extremal problem in graph theory (in Hungarian)},
		\newblock{\emph{Mat. Fiz. Lapok}} \textbf{48} (1941), 436--452.

        \bibitem{yuster2021factors}
		R. Yuster,
		\newblock{On factors of independent transversals in $ k $-partite graphs},
		\newblock{\emph{Electron. J. Comb.}} \textbf{28(4)} (2021), P4.23.
		
	\end{thebibliography}
\end{document}